\documentclass{amsart}

\usepackage{graphicx}
\usepackage{amssymb}
\usepackage[all]{xy}

\usepackage{lineno}
\setpagewiselinenumbers

 \newtheorem{theorem}{Theorem}[section]
 
 \newtheorem{corollary}[theorem]{Corollary}

 \newcommand{\Ker}{\mathop{\rm Ker}\nolimits}

 \newcommand{\RR}{{\mathbb{R}}}

\begin{document}

\title{GEOMETRY OF COMPACT LIFTING SPACES}

\author[G. Conner]{Gregory R. Conner$^1$}
\address{ 
Department of Mathematics, 
Brigham Young University, 
Provo, UT 84602, USA} 
\email{conner@mathematics.byu.edu}

\author[W. Herfort]{Wolfgang Herfort}
\address{ 
Institute for Analysis and Scientific Computation  
Technische Universit\"at Wien 
Wiedner Hauptstra\ss e 8-10/101
Vienna, Austria} 
\email{wolfgang.herfort@tuwien.ac.at}

\author[P. Pave\v si\'c]{Petar Pave\v si\'c $^2$}
\address{ 
Faculty of Mathematics and Physics  
University of Ljubljana 
Jadranska 21
Ljubljana, Slovenia} 
\email{petar.pavesic@fmf.uni-lj.si}
\thanks{$^1$ The first author is supported by Simons Foundation collaboration grant 246221 and $^2$ the third author was partially supported by the Slovenian Research Agency grant P1-0292
and the Slovenian Research Agency project J1-7025.}


\date{\today}

\begin{abstract}

 We study a natural generalization of inverse systems of finite regular covering spaces. A limit of such a system is a fibration whose fibres are profinite topological groups. However, as shown in \cite{{Conner-Herfort-Pavesic}}, there are many fibrations whose fibres are profinite groups,
which are far from being inverse limits of coverings. We characterize profinite fibrations among a large class of fibrations and relate the profinite topology on the fundamental group of the base with the action of the fundamental group on the fibre, and develop a version of the Borel 
construction for fibrations whose fibres are profinite groups. 
\ \\[3mm]
\end{abstract}

\maketitle


\ \\[-20mm]

\section{Introduction} 
A foundational result states that every Hausdorff, compact and totally disconnected group is \emph{profinite},
i.e. is an inverse limit finite groups. We extend this theorem to fibrations and relate local properties of a fibration with profinite topological groups as fibres,  to its global properties, namely being an inverse limit of finite coverings. 

An action of a topological group on a fibration can paint an intimate picture of the group, even more so than the analogous covering space action of a discrete group (consider the action of a $p$-adic group on a solenoid). 
Our main result will make this relation clearly visible. 

The notion of finite covering spaces is a basic tools in a wide array of mathematical topics including group theory, algebraic geometry, combinatorics,
 manifold theory and analysis.  One can generalize the notion of covering space by considering Hurewicz fibrations with unique path lifting, however this would allow for pathological fibers, such as the psuedocircle, and does not reveal a natural canonical group action like a covering space. Alternatively, one could consider inverse limits of finite covering spaces. These are very amenable to study, but seem to form only a very small subset of the former.  Our work shows that under reasonable hypotheses these two notions coincide.

{\bf Theorem} {\it Let $X$ be a path-connected and locally contractible space with a residually finite fundamental group
and let $p\colon E\to X$ be a fibration whose fibres are 
compact and totally path-disconnected. Then $p$ is an inverse limit of finite regular coverings if, and only if it is regular, $\pi_1$-profinite and has a dense leaf.}

In the rest of this introductory section we will explain some topological background about inverse limits of coverings and related fibrations with unique path-lifting property (which we call \emph{lifting projections} for short). 
Section 2 is preparatory and contains basic terminology and general 
facts about lifting projections. In Section 3 we prove the first 
basic result that compactness of the fibres and the transitive action
of the group of deck transformation already imply that the fibres of 
a lifting projection can be endowed with a natural profinite group
structure. In Section 4 we discuss and carefully compare the two ways 
in which the fundamental group of the base acts on the fibres of 
a lifting projection, and relate them to the profinite topologies
on the fundamental group itself. Finally, in Section 5 we develop 
a version of the 
Borel construction for lifting projections and use it as a tool to pass
from intrinsic properties of a lifting projection to an explicit system of finite coverings converging to it. 

Let $X$ be a topological space (that we assume path-connected for convenience), and let 
$$\xymatrix{
\widetilde{X}_1  \ar[d]_{p_1} & \widetilde{X}_2 \ar[l] \ar[d]_{p_2} & \widetilde{X}_3 \ar[l] \ar[d]_{p_3} & \cdots\ar[l]\\
X \ar@{=}[r] & X \ar@{=}[r] & X \ar@{=}[r] & \cdots }$$
be an inverse sequence of covering projections over $X$. The inverse limit of the sequence
may not be a covering projection in general, but it is always a Hurewicz fibration with unique path-lifting property and with totally path-disconnected fibres. 
These properties were established by Spanier \cite[Chapter II]{Spanier} who showed that Hurewicz fibrations with unique path-lifting 
property share many common features with covering projections, including monodromy theorems, relation with the fundamental group, lifting criteria
for maps to the base, properties of deck-transformations and many others. 
However, the well-known classification in terms of subgroups of the fundamental group of the base is available only for covering projections (cf. \cite[Section II.5]{Spanier}) 
and does not extend to the more general setting. 

Hurewicz fibrations with unique liftings of paths (which we call simply \emph{lifting projections}) represent a natural replacement for the concept of covering projections over base spaces with bad local properties (as compared with more involved generalizations like
Fox's overlays \cite{Fox}, generalized coverings in the sense of Fischer and Zastrow \cite{Fisher-Zastrow}, and another variant
introduced by Brodskiy, Dydak, Labuz and Mitra \cite{BDLM} that work well in the context of locally path-connected spaces). Fibrations with unique path liftings are implicitly used in Cannon and Conner \cite[Section 6]{Cannon-Conner} in the form of the action of the topological 'big' free group on its 'big' Cayley graph. However, greater
generality can lead to certain pathologies that we presented in 
\cite{Conner-Herfort-Pavesic} to explain some of the difficulties that one must face in order 
to build a reasonable classification theory of lifting projections.


\section{Lifting projections}

In this section we  recall the definition and briefly summarize some of the main properties of lifting projections. 
As we already mentioned, lifting projections were introduced by Spanier \cite{Spanier}, who called them (Hurewicz) fibrations with unique 
path-lifting property. 
In fact, Spanier developed most of the theory of covering spaces in this more general setting. For our purposes Hurewicz fibrations 
are most conveniently described in terms of lifting functions. 
Every map $p\colon E\to X$ induces a map $\overline p\colon E^I\to  E\times X^I$, $\overline p\colon \alpha\mapsto (\alpha(0),p\circ\alpha)$. 
In general $\overline p$ is not surjective, in fact its image is the subspace 
$$E\sqcap X^I:=\{ (e,\gamma)\in E\times X^I \mid p(e)=\alpha(0)\} \subset E\times X^I.$$ 
A \emph{lifting function} for $p$ is a section of $\overline p$, that is, a map $\Gamma\colon E\sqcap X^I\to E^I$ such that 
$\overline p\circ\Gamma$ is the identity map on $E\sqcap X^I$. Then we have the following basic characterization
(cf. \cite[Theorem 1.1]{Pavesic-Piccinini}): a map $p\colon E\to X$ is a \emph{Hurewicz fibration} if and only if it admits a continuous 
lifting function $\Gamma$. Furthermore, the unique path-lifting property of $p$ means then for every $\gamma\in E^I$ we have, $\Gamma(\gamma(0),p\circ\gamma)=\gamma$,
therefore $\Gamma$ is surjective and hence $\overline p$ is injective.

Thus we give the following definition:
a map $p\colon E\to X$ is a \emph{lifting projection} if $\overline p\colon E^I\to E\times X^I$ is an embedding. 
A \emph{lifting space} is a triple $(E,p,X)$ where $X$ is the \emph{base}, $E$ is the \emph{total space} and 
$p\colon E\to X$ is a lifting projection. We will occasionally abuse the terminology and refer to the space $E$ or the map $p$ itself as a lifting space over $X$.

For every $x\in X$ the subspace $p^{-1}(x)\subset E$ is called the \emph{fibre} of $p$ over $x$. If $X$ is path-connected, then all fibres
of $p$ are homeomorphic, so we will usually talk about \emph{the} fibre of $p$. A lifting projection is said to be \emph{compact} if 
its fibres are compact.

Clearly, every covering projection is a lifting projection. In fact, we may rephrase the definition of covering projections as locally trivial 
(i.e. evenly covered) lifting projections whose fibres are discrete. 

Conversely, if $p\colon E\to X$ is a lifting projection, where $X$ is locally path-connected and semi-locally simply connected, and $E$ is locally path-connected, 
then $p$ is a covering projection (\cite[Theorem II.4.10]{Spanier}). 
We obtain new examples of lifting projections as soon as we consider base spaces that are not semi-locally simply-connected (e.g., the projection 
$p\colon \RR^\infty\to (S^1)^\infty$) or total spaces that are not locally path-connected (e.g., the projection of the dyadic solenoid
onto the circle). 

One of the main advantages of lifting projections with respect to coverings is that they are preserved by arbitrary compositions, products and 
inverse limits (\cite[Theorem II.2.6 and II.2.7]{Spanier}). 

A Hurewicz fibration is a lifting projection if, and only if its fibres are totally path-disconnected (\cite[Theorem II.2.5]{Spanier}). 
Then the exact homotopy sequence of a fibration for a (based) lifting space $p\colon (E,e_0)\to (X,x_0)$ with fibre $F=p^{-1}(x_0)$ yields 
the exact sequence of groups and pointed sets
$$ 1\rightarrow\pi_1(E,e_0)\stackrel{p_\sharp}{\rightarrow} \pi_1(X,x_0)\rightarrow \pi_0(F)\rightarrow \pi_0(E)\rightarrow\ast$$
and the isomorphisms $\pi_n(E,e_0)\cong\pi_n(X,x_0)$ for $n\ge 2$. In particular we see that to every lifting projection over $X$
there corresponds a subgroup $\mathrm{Im}p_\sharp\le\pi_1(X,x_0)$, and this correspondence is fundamental to the classification of covering projections.

Another important property of covering projections that extends to lifting projections is the \emph{lifting criterion}: 
if $p\colon (E,e_0)\to (X,x_0)$ is a lifting projection, then a map $f\colon (Y,y_0)\to (X,x_0)$ from a connected and locally 
path-connected space $Y$ can be lifted to a map $\tilde f\colon (Y,y_0)\to (E,e_0)$ if, and only if 
$f_\sharp(\pi_1(Y,y_0))\le p_\sharp(\pi_1(E,e_0))$ (\cite[Theorem II.4.5]{Spanier}). Note however the critical assumption 
that the domain $Y$ is locally path-connected. In fact, the lack of this property in total spaces of lifting projections turns out 
to be the main obstruction for the extension of the classification theory of covering projections to general lifting projections. 

In \cite[Section II.5]{Spanier} Spanier describes the classification theory of covering projections. He first notes that 
the set of all covering projections over $X$ together with fibre-preserving maps between them forms a lattice.
Then he associates to every open cover $\mathcal{U}$ of $X$ a normal subgroup $\pi_1(X,\mathcal{U})\le\pi_1(X,x_0)$ defined as the normal closure
of the group of all $\mathcal{U}$-small loops. Spanier proves that if $X$ is a connected and locally path-connected space, then the correspondence 
$p\mapsto \mathrm{Im}p_\sharp$ determines a bijection between the lattice of all (pointed) covering projections over $X$ and the lattice of subgroups 
of $\pi_1(X,x_0)$ that contain $\pi_1(X,\mathcal{U})$ for some $\mathcal{U}$. 
In particular, if $\pi_1(X,\mathcal{U})=0$ for some $\mathcal{U}$ (such $X$ is said to be \emph{semi-locally simply-connected})
then the lattice of all covering projections over $X$ corresponds exactly to the lattice of all subgroups of $\pi_1(X,x_0)$. 

Very little of the classification theory remains valid for general lifting projections. One can define the universal lifting projection $p\colon\widehat X\to X$ as 
the initial object in the category of connected lifting spaces over $X$ but its total space $\widehat X$  is in general not simply-connected. In fact, one can
show that its fundamental 
group equals the intersection of all subgroups of $\mathcal{U}$-small loops (\cite[Theorem 3.5]{Conner-Pavesic}). Furthermore, the correspondence
$p\mapsto \mathrm{Im}p_\sharp$ is far from being injective as shown by the $p$-adic solenoids (inverse limits of $p$-fold coverings over the circle) 
which form an infinite family of simply-connected and yet non-isomorphic lifting projections over the circle.


\section{Fibres of regular compact lifting projections}

In the theory of covering spaces a role of particular importance is played by regular coverings. The most commonly used definition is that a covering projection
$p\colon\widetilde X\to X$ is regular if for every loop $\alpha$ in $X$ either every lifting of $\alpha$ to $\widetilde X$ is a loop or none 
of the liftings are loops. Spanier \cite{Spanier} uses the same formulation to define regularity for arbitrary lifting projections, and then proves that 
if $E$ is path-connected, then the lifting projection $p\colon E\to X$ is regular if and only if $p_\sharp(\pi_1(E))$ is a normal subgroup
of $\pi_1(X)$ (\cite[Theorem 2.3.12]{Spanier}). Furthermore, if $E$ is path-connected and locally path-connected then $p$ is regular if and only 
if the group of deck transformations $A(p)$ acts transitively on the fibres of $p$ (\cite[Corollary 2.6.3]{Spanier}). In general, transitive action
of $A(p)$ on the fibres of $p$ implies regularity, which in turn implies that the image of $p_\sharp$ is a normal subgroup of $\pi_1(X)$, but without additional assumptions
none of the implications can be reversed. Since in our considerations the transitive action of deck transformation play a much more important role than
the lifting of loops, we will depart from Spanier's terminology, and will define a lifting projection $p\colon E\to X$ to be \emph{regular} if $A(p)$ acts
transitively on the fibres of $p$.

A natural source of regular lifting projections are inverse limits of regular coverings. Let
$$\xymatrix{
\widetilde{X}_1  \ar[d]_{p_1} & \widetilde{X}_2 \ar[l] \ar[dl]_{p_2} & \widetilde{X}_3 \ar[l] \ar[dll]_{p_3} & \cdots\ar[l] & 
\widetilde{X}=\varprojlim X_i \ar[l] \ar[dllll]^{p}\\
X}$$
be an inverse sequence of regular covering spaces over $X$. Then the fibre $F$ of the lifting space $p\colon\widetilde X\to X$  
is homeomorphic to the inverse limit of groups $\varprojlim \pi_1(X)/\mathrm{Im}(p_i)_\sharp$ (\cite[Proposition 4.2]{Conner-Pavesic}), all leafs (path-components)
of $\widetilde X$ are dense in $\widetilde X$ (\cite[Proposition 4.7]{Conner-Pavesic}), and the  group of deck-transformations $A(p)$ acts freely and transitively on 
$F$ (\cite[Corollary 4.11]{Conner-Pavesic}). 

Is every regular lifting space with dense leaves isomorphic to an inverse limit of coverings? Not quite, but we can prove the remarkable fact that if the fibres are compact, then
they are actually inverse limits of finite groups.  
Toward the proof, let $p\colon E\to X$ be a regular lifting projection, and assume that at least one leaf (i.e path-component) of $E$ is dense in $E$. 
Fix base-points $x_0\in X$ and $\widetilde x_0\in F\cap L$ where $F=p^{-1}(x_0)$ is a fibre of $p$ and $L$ is a dense leaf in $E$. 
Then a deck transformation $\varphi\in A(p)$ is uniquely determined by its restriction $\varphi|_L$, which is in turn uniquely determined by the value 
$\varphi(\widetilde x_0)\in F$. By regularity, the correspondence $\varphi\mapsto \varphi(\widetilde x_0)$ defines a bijection $\Theta\colon A(p)\to F$. We may thus think of $A(p)$ as  
a group endowed with the subspace topology of $F$  in $E$ or alternatively, we may consider $F$ to be a topological space with a group structure induced by the bijection
$\Theta$. One is clearly interested when the two structures are compatible.

\begin{theorem}
\label{thm:fibres}
Assume that $p\colon E\to X$ is a regular compact lifting projection with a dense leaf. Then the fibre $F$ of $p$, viewed  as a subspace of $E$ and with 
the group structure determined by the bijection $\Theta\colon A(p)\to F$, is a compact, totally disconnected topological group.
\end{theorem}
\begin{proof} We are going to show that the group operation in $A(p)$ is left and right continuous with respect 
to the topology on $A(p)$ induced by the bijection $\Theta$. 

Fix some $\varphi\in A(p)$. If $\psi$ and $\psi'$ are close in $A(p)$ (which by definition means that $\psi(\widetilde x_0)$ and $\psi'(\widetilde x_0)$ are close in $F$),
then the continuity of $\varphi$ implies that $\varphi(\psi(\widetilde x_0))$ is close $\varphi(\psi'(\widetilde x_0))$. Therefore $\varphi\circ\psi$ is close
to $\varphi\circ\psi'$, and so the left multiplication in $A(p)$ is continuous.

Next, fix some $\psi\in A(p)$ and consider a sequence $\varphi_1,\varphi_2,\ldots $ in $A(p)$, converging to $\varphi$ (that is to say, the sequence
$\varphi_1(\widetilde x_0),\varphi_2(\widetilde x_0),\ldots$ converges to $\varphi(\widetilde x_0)$ in $F$). 
If $\psi(\widetilde x_0)\in L$ then there is a path $\widetilde\alpha\colon (I,0,1)\to (L,\widetilde x_0,\psi(\widetilde x_0))$ and the deck transformation $\varphi_i\circ\psi$ 
is determined by the relation 
$$\varphi_i(\psi(\widetilde x_0))=\Gamma(\varphi_i(\widetilde x_0),p\widetilde\alpha)(1).$$ 
Since the lifting is along the fixed path $p\widetilde\alpha$ and since the lifting function $\Gamma$ is continuous, it follows that 
the sequence $\Gamma(\varphi_i(\widetilde x_0),p\widetilde\alpha)(1)$ converges to $\Gamma(\varphi(\widetilde x_0),p\widetilde\alpha)(1)$ in $F$, therefore
the sequence $\varphi_i\circ\psi$ converges to $\varphi\circ\psi$ in $A(p)$. 

If $\psi(\widetilde x_0)\notin L$ then we may choose a sequence of elements $\psi_1,\psi_2,\ldots$ in $A(p)$ such that $\psi_1(\widetilde x_0),\psi_2(\widetilde x_0),\ldots$ are in the 
dense leaf $L$ and converge to $\psi(\tilde x_0)$. Then we can combine the continuity of the left-multiplication in $A(p)$ with the continuity of the right-multiplication 
by deck transformations that correspond to elements in $L$ to perform the following computation in the compact metric space $F$: 
$$\lim_i \varphi_i(\psi(\tilde x_0))=\lim_i \lim_j \varphi_i(\psi_j(\tilde x_0))=\lim_j \lim_i \varphi_i(\psi_j(\tilde x_0))=\lim_j \varphi_i(\psi(\tilde x_0))=
\varphi(\psi(\tilde x_0)).$$
We conclude that the multiplication in $A(p)$ is right-continuous as well. 
By a theorem of R. Ellis \cite{Ellis57} a left and right-continuous compact group is actually a topological group.

We may now use the topological group structure to show that $F$ is totally disconnected. It clearly suffices to consider the unit component of $G$, so
we may assume without loss of generality that $G$ is connected, and show that $G=1$. 
By \cite[9.36(ix)]{hofmor} every element $x\in G$ belongs to a compact connected monothetic subgroup, say $T$, of $G$. 
If we can prove that $T=1$ we are done. Now, $T$ is connected and monothetic and hence is an epimorphic image of the
universal monothetic group $\mathbb S$. For the latter (and actually for any connected compact monothetic group) it
is known that every element belongs to a one-parameter subgroup, that is, it is path connected. Therefore our $T$ is path connected and trivial, as wanted.
\end{proof}

By a famous theorem on topological groups every compact, Hausdorff and totally disconnected topological group is a \emph{profinite} topological group, i.e., 
it can be obtained as an inverse limit of a sequence of finite groups, see \cite[Theorem 1.34]{hofmor}. Thus, we get the following 

\begin{corollary}
\label{cor:profinite fibres}
If $p\colon E\to X$ is a regular compact lifting projection with a dense leaf, then its fibres are profinite topological groups.
\end{corollary}

One reason for the importance of the above result is that it rules out the possibility that $F$ is a totally path-disconnected but 
not totally disconnected continuum (for example a pseudo-arc or pseudo-circle). 
In fact, since we assumed that $F$ is a metric space, 
we may conclude that $F$, as a topological space, is homeomorphic to the Cantor set. 

Even when we know that the fibres of a lifting projection over $X$ are inverse limits of finite groups, we cannot conclude that the lifting 
projection can be obtained as an inverse limit of finite regular coverings, simply because the fibre need not be
an inverse limit of finite quotients of $\pi_1(X)$. It is clear that the topology of the fibre must be in some sense compatible with 
the family of finite quotients of the fundamental group of the base.

Let us say that a lifting projection $p\colon \widetilde X\to X$ is \emph{profinite} if it can be obtained as the inverse limit of a system 
of regular, finite-sheeted covering projections $\{p_i\colon \widetilde X_i\to X\}$. If we assume in addition that 
the intersection $\bigcap_i \mathrm{Im}(p_i)_\sharp$ is trivial, then the canonical map $\pi_1(X)\to\varprojlim \pi_1(X)/\mathrm{Im}(p_i)_\sharp$ 
is injective and can be viewed as the profinite completion of $\pi_1(X)$ with respect to the family of subgroups $\{\mathrm{Im}(p_i)_\sharp\}$. 
If $\bigcap_i \mathrm{Im}(p_i)_\sharp$ is not trivial, then the canonical map is not injective but is still continuous if we consider
$\pi_1(X)$ as a topological space with the profinite topology (the group topology generated by the family of all normal subgroups of finite index).
This observation will be used in the following section to formulate a  suitable compatibility condition between  the topology of the fibre and 
the fundamental group of the base.

\section{Two actions}

In the proof of Theorem \ref{thm:fibres} we used separate arguments to prove the continuity of the left and of the 
right multiplication in $A(p)$. This is a reflection of the  fact that there are actually two fundamentally different ways 
in which $\pi_1(X,x_0)$ acts on the fibre $F:=p^{-1}(x_0)$ of a lifting space $p\colon E\to X$.
In this section we study the properties of these actions and the relations between them. 

Given a path  $\alpha\colon (I,0,1)\mapsto (X,x_0,x_1)$ and a point $\tilde x\in p^{-1}(x_0)$, we may use the lifting function for the lifting space $p\colon E\to X$ to 
define an operation of $\alpha$ on $E$ as
$$\tilde x\alpha:=\Gamma(\tilde x,\alpha)(1)\in p^{-1}(x_1).$$ 
By the monodromy theorem the value $\tilde x\alpha$ depends only on the homotopy class (rel $\partial I$) of $\alpha$. 
Moreover, for a fixed $\alpha$ the map $\tilde x\mapsto \tilde x\alpha$ is continuous, and when $\alpha$ and $\beta$ can be 
concatenated we have the associativity relation
$$(\tilde x\alpha)\beta=\tilde x(\alpha\cdot\beta).$$  
The case $x_0=x_1$ is of special interest, because then we obtain a \emph{right action} of $\pi_1(X,x_0)$ on $F$.

While the definition of the right action is completely general as it depends only on the fibration property of lifting spaces, 
in order to define a left action of $\pi_1(X)$ on $F$ the lifting space must 
satisfy certain additional assumptions. Let $L$ be a leaf of $E$ and let  $\widetilde x_0\in F\cap L$.
Assume that the group of deck transformations $A(p)$ acts transitively on $F\cap L$, and that for any $\varphi,\psi\in A(p)$ the condition 
$\varphi(\tilde x_0)=\psi(\tilde x_0)$ 
implies $\varphi=\psi$  (this is the case, for example, if $L$ is dense $E$). Then for every $\alpha\in\pi_1(X,x_0)$ there exists a unique 
$\varphi_\alpha\in A(p)$ such that 
$\varphi_\alpha(\tilde x_0)=\tilde x_0\alpha$. The value of $\varphi_\alpha$ on an element $\tilde x\in L$ can be expressed in terms of the right action:
$\varphi_\alpha(\tilde x)=(\tilde x_0 \alpha) \gamma$, where $\gamma$ is any path in $X$ such that $\tilde x_0\gamma=\tilde x$. Thus we may introduce the notation 
$$\alpha\tilde x:=\varphi_\alpha(\tilde x).$$ 
Moreover, we have $\beta\tilde x=(\tilde x_0 \beta)\gamma=\tilde x_0(\beta\cdot\gamma)$, and hence 
$$\alpha(\beta \tilde x)=(\tilde x_0 \alpha)(\beta\cdot\gamma)=(\tilde x_0(\alpha\cdot\beta))\gamma)=(\alpha\cdot\beta)\tilde x.$$
We conclude that under the required assumptions we obtain a \emph{left action} of $\pi_1(X,x_0)$ on $F$. 
Furthermore, for $\alpha,\beta\in\pi_1(X,x_0)$ and $\tilde x\in F$ we have $\tilde x\beta=\tilde x_0(\gamma\cdot\beta)$, which yields
the associativity relation
$$\alpha(\tilde x\beta)=(\tilde x_0\alpha)(\gamma\cdot\beta)=((\tilde x_0)\alpha)\gamma)\beta=(\alpha\tilde x)\beta.$$

To  clarify the difference between the two actions, let us consider the infinite 4-valent tree (figure \ref{fig:tree}) viewed as the universal covering of the wedge of two circles
$S^1\vee S^1$.
The fibre over the base-point of the wedge is represented by the crossings in the tree, while the fundamental group of the wedge is isomorphic to $F_2$, the free group on two generators. 
Let $\alpha\in\pi_1(S^1\vee S^1,x_0)$ be the generator represented by the projection of the path $\widetilde\alpha$ from $\tilde x_0$ to $\tilde x_1$ in the tree. 
\begin{figure}[ht]
    \centering
    \includegraphics[scale=0.4]{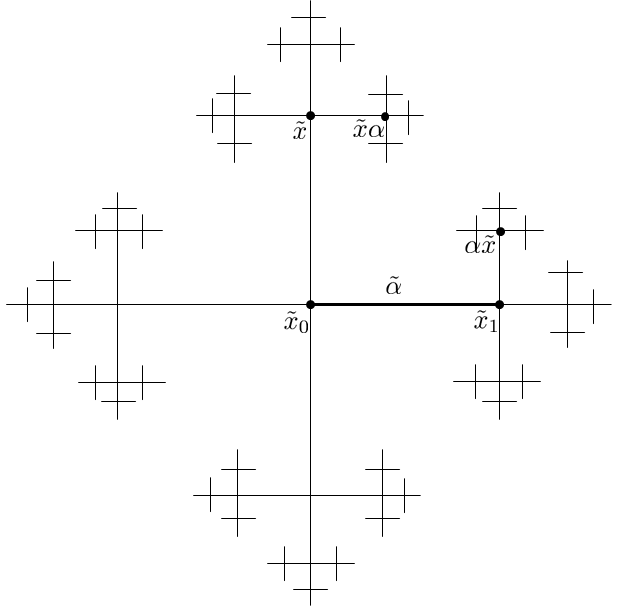}
    \caption{Left and right action.}
    \label{fig:tree}
\end{figure}
Then the right action of $\alpha$ on the fibre simply shifts each node in the tree by one segment to the left. The left action is slightly harder to visualize because it is given by 
the unique deck transformation of the tree that maps $\tilde x_0$ to $\tilde x_1$. Observe that the points in the fibre on which the left and the right action of $\alpha$ coincide
are exactly those that lie on the horizontal line through $\tilde x_0$ and $\tilde x_1$.

Both actions play an important role in the theory of lifting spaces. We give a list of their properties, that are closely related and in some sense complementary. 
\begin{enumerate}
\item In general, neither of the actions is faithful. In fact, for $\alpha\in p_\sharp(\pi_1(E,\tilde x_0))$ we have $\tilde x\alpha=\tilde x$ for all $\tilde x\in F$,
and $\alpha\tilde x=\tilde x$ for all $\tilde x\in F\cap L$. In other words, both actions can be reduced to faithful actions of $\pi_1(X,x_0)/p_\sharp(\pi_1(E,\tilde x_0))$
on $F$.
\item The definition of the right action is intrinsic, while the definition of the left action heavily depends on the choice of the base-point $\tilde x_0$. 
Both actions agree on $\tilde x_0$, as $\alpha\tilde x_0=\tilde x_0\alpha$ for all $\alpha$, and normally disagree on other points in $F$. However, if $\alpha$ is in
the center of $\pi_1(X,x_0)$ then we have 
$$\alpha\tilde x=\tilde x_0 (\alpha\cdot\gamma)=\tilde x_0 (\gamma\cdot\alpha)=(\tilde x_0\gamma)\alpha=\tilde x\alpha.$$
In particular, left and right action agree if the fundamental group of the base is abelian.
\item 
The left action, whenever defined, is actually a restriction of some deck transformation, which is a continuous self-map of the entire total space. On the other hand,
the right action on the fibre, cannot be extended to a continuous self-map of  $E$, unless of course, it coincides with the left action. Nevertheless, when
the left action is defined, we may use it to determine the right action by the formula
$$\tilde x\alpha:=(\gamma\cdot\alpha)\tilde x_0, \ \ \text{where} \  \ \gamma\tilde x_0=\tilde x.$$
\item
The action of the group of deck transformations $A(p)$ on $E$  naturally restricts to a left action of $A(p)$ on the fibre $F$. Assume that the action on $F$ is 
transitive, and that $E$ has a dense leaf $L$. By choosing $\tilde x_0\in F\cap L$ we may define a left action of $\pi_1(X,x_0)$ on $F$.
This action is compatible with the action of $A(p)$ on $F$, with respect to the homomorphism $\alpha\mapsto\varphi_\alpha$ which is uniquely determined by 
the relation $\alpha\tilde x_0=\varphi_\alpha(\tilde x_0)$ as previously discussed. In fact, under the above assumptions, there is a bijection 
$\Theta\colon A(p)\to F$. It is easy to verify that with respect to this identification the left and right actions of $\pi_1(X,x_0)$ on $F$ exactly correspond 
to the left and right regular actions of $A(p)$ on itself.
\end{enumerate}

The right action of $\pi_1(X,x_0)$ on $F$ determines a map 
$$\theta\colon\pi_1(X,x_0)\to F,\;\;\;\;\; \alpha\mapsto x_0\alpha.$$
We will say that a compact lifting space $p\colon E\to X$ is \emph{$\pi_1$-profinite} if  $\theta$ is continuous with respect to the 
profinite topology on the fundamental group. Clearly, every profinite lifting projection is also $\pi_1$-profinite.

\section{Characterization of profinite lifting projections}

As we explained above, every profinite lifting projection over $X$ is regular, $\pi_1$-profinite and with dense leaves. We are now ready to prove 
a partial converse to this result. In the process we will need to assume that the base space is locally nice, 
which is not surprising in view of the vast variety of anomalous examples of lifting spaces (cf. in particular 
Example 2.7 in \cite{Conner-Herfort-Pavesic}). 

Let $X$ be a connected and locally path-connected space with a residually finite fundamental group. Denote by 
$(\widehat X,\widehat p,X)$ the inverse limit of all regular, finite-sheeted coverings of $X$. 
The leaves of $\widehat X$ are simply-connected and dense in $\widehat X$, and the fibre of $\widehat p$ can be naturally 
identified with the group of deck transformations $A(\widehat p)$,
which can in turn be identified with the profinite completion of $\pi_1(X)$. 
Fix base-points $x_0\in X$ and $\widehat x_0\in \bar p^{-1}(x_0)$, and denote by  $L$ the leaf 
of $\widehat X$ containing $\widehat x_0$. Then the restriction $q\colon L\to X$ is a lifting space, whose 
fibre $q^{-1}(x_0)=\widehat p^{-1}(x_0)\cap L$ 
can be identified with $\pi_1(X,x_0)$ through the action map $\alpha\mapsto \bar x_0\alpha $. On the other hand, we may also
identify $\pi_1(X,x_0)$ with the group of
deck transformations $A(q)\le A(\widehat p)$, and view $L$ as a left $\pi_1(X,x_0)$-space. Observe that the lifting space $q\colon L\to X$ is 
completely determined by $X$, 
and that it is actually a variant of Spanier's construction of the universal lifting space (cf. \cite[Section II.5]{Spanier}).

Given an arbitrary lifting projection $p\colon E\to X$, its fibre $F:=p^{-1}(x_0)$ is a right $\pi_1(X,x_0)$-space and we can define a map 
$$\Phi\colon F\times L\to E$$ 
as follows. For every $\widehat x\in L$ choose a path $\alpha\colon (I,0)\to (X,x_0)$ such that $\widehat x_0\alpha=\widehat x$ 
and let 
$$\Phi(u,\widehat x):=u\alpha.$$ 
Since $L$ is simply connected, if $\alpha'$ is another path satisfying $\widehat x_0\alpha'=\widehat x$, then 
$\alpha\simeq\alpha'(\mathrm{rel}\,\partial I)$, and so $u\alpha=u\alpha'$.
Moreover, $\Phi$ is clearly surjective, because every point in $E$ is connected by a path to some point in $F$.
On the other hand, $\Phi$ is not injective. In fact, if $\Phi(u,\widehat x)=\Phi(v,\widehat y)$, then one can find paths 
$\alpha,\beta\colon (I,0)\to (X,x_0)$ such that $u\alpha=v\beta$, therefore
$u(\alpha\beta^{-1})=v$. Furthermore, since the left and the right action agree on the base point, we also have 
$\alpha\widehat x_0=\widehat x$ and $\beta\widehat x_0=\widehat y$, 
hence $\widehat x=(\alpha\beta^{-1})\widehat y$. As a consequence  $\Phi(u,\widehat x)=\Phi(v,\widehat y)$ if, and only if there exists 
an element $\gamma\in\pi_1(X,x_0)$ such that 
$u\gamma=v$ and $\widehat x=\gamma\widehat y$.
Let us define an equivalence relation on $F\times L$ by requiring that $(u\gamma,\widehat y)\sim (u,\gamma\widehat y)$ for 
$\gamma\in\pi_1(X,x_0)$ 
(note that the left action of $\gamma$ is defined on any element of $L$, not only on the fibre over $x_0$), 
and define the Borel construction $F\times_{\pi_1(X)}L:=F\times L/\sim$. 
Then $\Phi$ induces a fibre-preserving bijection
$$\overline\Phi\colon F\times_{\pi_1(X)}L\to E.$$

\begin{theorem}
\label{thm:profinite}
Let $X$ be a path-connected and locally categorical space with a residually finite fundamental group.
If $p\colon E\to X$ is a regular, $\pi_1$-profinite lifting projection with dense leaves, then the map 
$$\overline\Phi\colon F\times_{\pi_1(X)}L\to E $$ 
is a homeomorphism.
\end{theorem}
\begin{proof}
Let $F_0=q^{-1}(x_0)=p^{-1}(x_0)\cap L$. Our first goal is to prove that the restriction 
$$\overline\Phi\colon F \times_{\pi_1(X)} F_0\to F$$ 
is a homeomorphism. Consider the following diagram, where $\pi_1(X)$ is endowed with the profinite topology, while 
$\theta_0\colon\pi_1(X)\to F_0$ 
and $\theta\colon\pi_1(X)\to F$ are defined respectively as $\theta_0(\alpha):=\widehat x_0\alpha$ and $\theta(\alpha)=\tilde x_0\alpha$.
Observe that $\theta_0$ is a homeomorphism and that $\theta$ is continuous.
Furthermore, let $\Theta\colon A(p)\to F$ be the homeomorphism from Theorem \ref{thm:fibres}, and let $\mu\colon F\times F \to F$ be the multiplication induced by the composition in $A(p)$.
$$\xymatrix{
F\times\pi_1(X) \ar[r]^-{1\times\theta_0}_-\approx \ar[d]_{1\times\theta} & F\times F_0 \ar[d]^\Phi \ar[r] & F\times_{\pi_1(X)} F_0 \ar[d]^{\overline\Phi}\\
F\times F \ar[r]^{\mu}& F \ar@{=}[r] & F \\
A(p)\times A(p)\ar[u]^{\Theta\times\Theta} \ar[r]_-\circ & A(p) \ar[u]_\Theta  }$$
We must first check that the upper left square is commutative. For $(u,\alpha)\in F\times\pi_1(X)$ we have 
$$\Phi((1\times\theta_0) (u,\alpha))=\Phi(u,\widehat x_0\alpha)=u\alpha.$$
On the other hand we can apply the formula $[\Theta^{-1}(u)](\widetilde x_0)=u$ and the associativity relation between the left and right action to compute
$$\mu((1\times\theta) (u,\alpha))=\Theta[\Theta^{-1}(u)\circ\Theta^{-1}(\widetilde x_0\alpha)]=\Theta^{-1}(u)[\Theta^{-1}(\widetilde x_0\alpha)(\widetilde x_0)]=$$
$$=\Theta^{-1}(u)(\widetilde x_0\alpha)=
(\Theta^{-1}(u)(\widetilde x_0))\alpha=u\alpha$$
The lifting projection $p$ is, by assumption, $\pi_1$-profinite, which means that the map $\theta$ is continuous, therefore the action map $F\times \pi_1(X)\to F$ is also continuous. Since $\pi_1(X)$ acts on $F$ 
by homeomorphsms, the action map is an open map. As a consequence, the map $\Phi\colon F\times F_0\to F$ is continuous and open, which in turn implies that $\overline\Phi\colon F \times_{\pi_1(X)} F_0\to F$ is 
a homeomorphism.

Recall (cf. \cite[Definition 5.12]{James}) that a space is \emph{locally categorical} if it can be covered by categorical open sets. Every lifting 
projection over a locally categorical space $X$ is locally trivial. In fact, if $U\subseteq X$ is a categorical subset with a contracting homotopy $H\colon U\times I\to U$ 
such that $H(U\times\{1\})=x_0$, then the map 
$$\xi\colon p^{-1}(U)\to U\times F,\;\;\; \xi(\widetilde u):=\big(p(\widetilde u), \Gamma(\widetilde u,H(p(\widetilde u),-)(1)\big)$$
is a fibre-preserving homeomorphism. Similarly, the restriction to $U$ of  the lifting
projection $q\colon L\to X$ is fibrewise homeomorphic to $U\times F_0$. As a consequence, we may identify 
the restriction to $U$ of the map $\overline \Phi\colon F\times_{\pi_1(X)} L\to F$ with the map 
$$1\times \overline\Phi\colon U\times (F\times_{\pi_1(X)} F_0)\to U\times F,$$
which is a homeomorphism by the discussion above.

Being a bijection and a local homeomorphism, the map $\overline\Phi\colon F\times_{\pi_1(X)} L\to E$ is  a homeomorphism.
\end{proof}

We may now prove a characterization of profinite lifting projections.

\begin{theorem}
Let $X$ be a path-connected and locally categorical space with a residually finite fundamental group.
A lifting projection $p\colon E\to X$ is profinite if, and only if it is regular, $\pi_1$-profinite and has a dense leaf.
\end{theorem}
\begin{proof}
In view of the above discussion it remains to prove the 'if' part. By Corollary \ref{cor:profinite fibres} the fibre $F$ of $p$ is a profinite topological group,
so there are normal subgroups of finite index $F_i \triangleleft F$ such that $F=\varprojlim F/F_i$. Let $\theta_i$ denote the composition of $\theta$ followed by
the quotient map
$$\theta_i\colon\pi_1(X)\stackrel{\theta}{\longrightarrow} F \longrightarrow F/F_i.$$
Since the image of $\theta$ consists of the points of $F$ that are contained in one leaf, and since the leaves of $p$ are dense in $E$, the image of $\theta$ is dense 
in $F$ which implies that $\theta_i$ is surjective. As a consequence $\Ker\theta_i$ is a normal subgroup of finite index in $\pi_1(X)$ and 
$$\pi_1(X)/\Ker\theta_i\cong F/F_i.$$
By considering the quotients $\pi_1(X)/\Ker\theta_i$ as right $\pi_1(X)$-spaces, the Borel construction yields finite-sheeted 
covering projections
$$p_i\colon (\pi_1(X)/\Ker\theta_i)\times_{\pi_1(X)} L\to X.$$
The covering projections $p_i$ form an inverse system and the canonical map 
$$F\times_{\pi_1(X)} L\to  \varprojlim\big[(\pi_1(X)/\Ker\theta_i)\times_{\pi_1(X)} L\big]$$
is clearly a fibre-preserving homeomorphism. In view of Theorem \ref{thm:profinite} we may conclude that the lifting projection $p\colon E\to X$ is equivalent to the 
inverse limit of finite-sheeted covering projections $p_i$.
\end{proof}

\end{document}